\documentclass[12pt]{amsart}

\usepackage{amsmath}
\usepackage{amssymb}
\usepackage{amsfonts}
\usepackage{amsthm}
\usepackage{enumerate}
\usepackage{hyperref}
\usepackage{color}

\theoremstyle{plain}
\newtheorem{thm}{Theorem}[section]

\newtheorem{prop}[thm]{Proposition}
\newtheorem{lem}[thm]{Lemma}
\newtheorem{cor}[thm]{Corollary}

\newtheorem{conj}[thm]{Conjecture}

\theoremstyle{definition}
\newtheorem{dfn}[thm]{Definition}

\newtheorem{exmp}[thm]{Example}

\newtheorem{rem}[thm]{Remark}

\newtheorem{dfns-rems}[thm]{Definitions and Remarks}
\newtheorem{notas-rems}[thm]{Notations and Remarks}
\newtheorem{exmps-rems}[thm]{Examples and Remarks}

\begin{document}


\title[Intersecting faces of a simplicial complex via algebraic shifting]{Intersecting faces of a simplicial complex via algebraic shifting}


\author[S. A. Seyed Fakhari]{S. A. Seyed Fakhari}

\address{S. A. Seyed Fakhari, Department of Mathematical Sciences,
Sharif University of Technology, P.O. Box 11155-9415, Tehran, Iran.}

\email{fakhari@ipm.ir}

\urladdr{http://math.ipm.ac.ir/fakhari/}


\begin{abstract}

A family $\mathcal{A}$ of sets is {\it $t$-intersecting} if the cardinality of the intersection of every pair of sets in $\mathcal{A}$ is at least $t$, and is an {\it $r$-family} if every set in $\mathcal{A}$ has cardinality $r$.
A well-known theorem of Erd\H{o}s, Ko, and Rado bounds the cardinality
of a $t$-intersecting $r$-family of subsets of an $n$-element set, or equivalently of $(r-1)$-dimensional faces of a simplex with $n$ vertices. As a generalization of the    Erd\H{o}s-Ko-Rado theorem, Borg presented a conjecture concerning the size of a  $t$-intersecting $r$-family of faces of an arbitrary simplicial complex. He proved his conjecture for shifted complexes. In this paper we give a new proof for this result based on work of Woodroofe. Using algebraic shifting we verify Borg's conjecture in the case of sequentially Cohen-Macaulay $i$-near-cones for $t=i$.
\end{abstract}








\maketitle


\section{Introduction} \label{sec1}

Throughout this paper, the set of positive integers $\{1, 2, \ldots\}$ is denoted by $\mathbb{N}$. For $m, n \in \mathbb{N}, m \leq n$, the set $\{i\in
\mathbb{N} : m\leq i\leq n\}$ is denoted by $[m, n]$; for $m = 1$, we also write $[n]$.

Let $t\leq r$ be two natural numbers. A family $\mathcal{A}$ of sets is {\it $t$-intersecting} if the cardinality of the intersection of every pair of sets in $\mathcal{A}$ is at least $t$, and is an {\it $r$-family} if every set in $\mathcal{A}$ has cardinality $r$.
A well-known theorem of Erd\H{o}s, Ko, and Rado bounds the cardinality
of a $t$-intersecting $r$-family:

\begin{thm} \label{ekr1}
Assume that $t\leq r$ are two natural numbers. Let $n\geq(t+1)(r-t+1)$ and $\mathcal{A}$ be a $t$-intersecting $r$-family of subsets of $[n]$. Then $|\mathcal{A}| \leq {n-t\choose r-t}$.
\end{thm}

Given a simplicial complex $\Delta$ (defined in Section \ref{sec2}) and a face $\sigma$ of
$\Delta$, we define the {\it link} of $\sigma$ in $\Delta$ to be

$${\rm lk}_{\Delta} \sigma = \{\tau : \tau \cup \sigma \in \Delta, \tau \cap \sigma =\emptyset\}.$$

Also for every integer $s\geq0$ we define

$$\Delta_{(s)}=\{\sigma \in \Delta : |\sigma|=s\}.$$

An $r$-face of $\Delta$ is a face of cardinality $r$. We further let $f_r(\Delta)$ be defined
as the number of $r$-faces in $\Delta$, and the tuple $(f_0(\Delta), f_1(\Delta),\ldots, f_{d+1}(\Delta))$
(where $d$ is the dimension of $\Delta$) is called the $f$-vector of $\Delta$.

{\bf Note}. We follow Swartz \cite{s} in our definition of $r$-face and $f_r$. Other
sources define an $r$-face to be a face with dimension $r$ (rather than
cardinality $r$) which shifts the indices of the $f$-vector by $1$.

We restate Theorem \ref{ekr1} using this language:

\begin{thm} \label{ekr}
Assume that $t\leq r$ are two natural numbers. Let $n\geq(t+1)(r-t+1)$ and $\mathcal{A}$ be a $t$-intersecting $r$-family of faces
of the simplex with $n$ vertices. Then $|\mathcal{A}|\leq f_{r-t}({\rm lk}_{\Delta} \sigma)$, where $\sigma$ is is a $t$-face of $\Delta$.
\end{thm}

\begin{dfn}
A simplicial complex $\Delta$ is called {\it ($t,r$)-EKR} if every $t$-intersecting $r$-family $\mathcal{A}$ of faces of $\Delta$ satisfies $|\mathcal{A}|\leq \max f_{r-t}({\rm lk}_{\Delta} \sigma)$, where the maximum is taken over all $t$-faces $\sigma$ of $\Delta$. Equivalently, $\Delta$ is ($t,r$)-EKR if the set of all $r$-faces containing some $t$-face $\sigma$ has
maximal cardinality among all $t$-intersecting families of $r$-faces.
\end{dfn}

As a generalization of the Erd\H{o}s-Ko-Rado theorem, Borg conjectured that:

\begin{conj} \label{con}
{\rm (}\cite[Conjecture 2.6]{b}{\rm )} Let $t\leq r$ be two natural numbers. Assume that $\Delta$ is a simplicial complex having
minimal facet cardinality $k \geq(t+1)(r-t+1)$ and suppose that  $S\neq \emptyset$ is a subset of $[t, r]$. Then every $t$-intersecting family $\mathcal{A}$
 of faces of $\Delta$ with $\mathcal{A}\subseteq \bigcup_{s\in S}\Delta_{(s)}$ satisfies the following inequality:
$$|\mathcal{A}| \leq \max \sum_{s\in S}f_{s-t}({\rm lk}_{\Delta}\sigma), \ \ \ \ \ \ \ \ \ \ \ \ (\ast)$$
where the maximum is taken all over $t$-faces $\sigma$ of $\Delta$.
\end{conj}

Borg proved Conjecture \ref{con} for shifted complexes \cite[Theorem 2.7]{b}. Using algebraic shifting, Woodroofe gave a new proof for \cite[Theorem 2.7]{b} in a special case of $t=1$ and $S=\{r\}$ \cite[Lemma 3.1]{w}. In this paper we extend Woodroofe's proof and give a complete new proof for \cite[Theorem 2.7]{b} using algebraic shifting (Theorem \ref{proof}). Woodroofe also proved, that in the special case of $t=1$ and $S=\{r\}$, Conjecture \ref{con} is true for sequentially Cohen-Macaulay near-cones \cite[Corollary 3.4]{w}. We also generalize this result and prove that Conjecture \ref{con} is true for every sequentially Cohen-Macaulay $i$-near-cone in the case of $t=i$ (Corollary \ref{scm}).

\begin{rem}
It was also proved by Borg \cite[Theorem 2.1]{b} that Conjecture \ref{con} is true when the minimum facet cardinality of $\Delta$ is at least $(r-t){3r-2t-1\choose t+1}+r$.
\end{rem}

This paper is organized as follows. In Section \ref{sec2} we review the necessary
background on shifted complexes, algebraic shifting, the Cohen-Macaulay property, and $i$-near-cones. In Section \ref{sec3} we present our new proof for \cite[Theorem 2.7]{b}. In Section \ref{sec4} we prove the main results of this paper about intersecting faces of $i$-near-cones, (Corollaries \ref{conje} and \ref{scm}).

\section{Algebraic shifting and near-cones} \label{sec2}

An {\it (abstract)} simplicial complex $\Delta$ on the set of vertices $V(\Delta)$ is a collection of subsets of $[n]$ which is closed under
taking subsets; that is, if $F \in \Delta$ and $F'\subseteq F$, then also
$F'\in\Delta$. Every element $F\in\Delta$ is called a {\it face} of
$\Delta$. We assume that every vertex is contained in some face. The {\it size} of a face $F$ is defined to be $|F|$ and its {\it
dimension} is defined to be $|F|-1$. (As usual, for a given finite set $X$,
the number of elements of $X$ is denoted by $|X|$.) The {\it dimension} of
$\Delta$ which is denoted by $\dim\Delta$, is defined to be $d-1$, where $d
=\max\{|F|\mid F\in\Delta\}$. . A
{\it facet} of $\Delta$ is a maximal face of $\Delta$ with respect to
inclusion. We say that $\Delta$ is {\it pure} if all facets of
$\Delta$ have the same cardinality

If $\mathcal{F}$ is some family of sets, then the simplicial complex $\Delta(\mathcal{F})$ {\it generated by} $\mathcal{F}$ has faces consisting of all subsets of all sets in $\mathcal{F}$. For a
simplicial complex $\Delta$, the $r$-{\it skeleton} $\Delta^{(r)}$ consists of all faces of $\Delta$ having
dimension at most $r$, while the {\it pure $r$-skeleton} is the subcomplex
generated by all faces of $\Delta$ having dimension exactly $r$. The {\it join} of
disjoint simplicial complexes $\Delta$ and $\Sigma$ is the simplicial complex $\Delta\ast\Sigma$
with faces $\tau\cup\sigma$, where $\tau$ is a face of $\Delta$ and $\sigma$ is a face of $\Sigma$.
For every vertex $\sigma \in \Delta$, {\it link} and {\it anti-star} of $\sigma$ are defined by
$${\rm lk}_{\Delta} \sigma =\{\tau \in \Delta : \tau \cap \sigma= \emptyset , \tau \cup \sigma \in \Delta\}$$
and
$${\rm ast}_{\Delta} \sigma =\{\tau \in \Delta : \tau \cap \sigma= \emptyset\}.$$

Also for every integer $s\geq0$ we define

$$\Delta_{(s)}=\{\sigma \in \Delta : |\sigma|=s\}.$$

A simplicial complex $\Delta$ with ordered vertex set $\{v_1, \ldots, v_n\}$ is {\it shifted}
if whenever $\sigma$ is a face of $\Delta$ containing vertex $v_i$, then $(\sigma \setminus \{v_i\})\cup \{v_j\}$
is a face of $\Delta$ for every $j < i$. An $r$-family $\mathcal{F}$ of subsets of $\{v_1, \ldots, v_n\}$
is {\it shifted} if it generates a shifted complex.
\subsection{Algebraic shifting}
A {\it shifting operation} on an ordered vertex set $V$ is a map which associates each simplicial
complex $\Delta$ on $V$ with a simplicial complex ${\rm Shift}\Delta$ on $V$ and which satisfies
the following conditions:

\begin{itemize}
\item[($S_1$)] ${\rm Shift}\Delta$ is a shifted.

\item[($S_2$)] If $\Delta$ is shifted, then ${\rm Shift}\Delta = \Delta$.

\item[($S_3$)] $f_i({\rm Shift}\Delta) = f_i(\Delta)$ for all $i$.

\item[($S_4$)] If $\Gamma\subseteq\Delta$ are simplicial complexes,
then ${\rm Shift}\Gamma\subseteq {\rm Shift}\Delta$.
\end{itemize}

If $\mathcal{A}$ is some $r$-family of sets, then ${\rm Shift}\mathcal{A}$ is defined to be
$${\rm Shift}\mathcal{A}=
{\rm Shift}(\Delta (\mathcal{A}))_{(r)}.$$

In our proofs we need a shifting operation which satisfies the following extra property:

\begin{itemize}
\item[($S_5$)] If $\mathcal{A}$ is a $t$-intersecting $r$-family, then
${\rm Shift}\mathcal{A}$ is a $t$-intersecting $r$-family.
\end{itemize}

Kalai proves (See \cite[Corollary 6.3 and subsequent Remarks]{k}) that a specific shifting operation which is called {\it exterior algebraic shifting} (with respect to a field $\mathbb{F}$) satisfies ($S_5$). We denote the exterior algebraic shift of $\Delta$, with respect to a field $\mathbb{F}$, by ${\rm Shift}_{\mathbb{F}} \Delta$. (The precise definition of exterior algebraic shifting will not be important for us, but can be found in
Kalai's survey article \cite{k}.)

\subsection{Near-cones}
A simplicial complex $\Delta$ is a {\it near-cone} with respect
to an {\it apex vertex} $v$ if for every face $\sigma$, the set $(\sigma \setminus \{w\})\cup \{v\}$ is also
a face for each vertex $w\in \sigma$. Equivalently, the boundary of every
facet of $\Delta$ is contained in $v\ast{\rm lk}_{\Delta} v$; another equivalent condition is
that $\Delta$ is the union of $v\ast{\rm lk}_{\Delta} v$ and some set of facets not containing $v$
(but whose boundary is contained in ${\rm lk}_{\Delta} v$). If $\Delta$ is a cone with apex
vertex $v$, then obviously $\Delta = v\ast{\rm lk}_{\Delta} v$, thus every cone is a near-cone.

\begin{dfn} \label{ncone}
A simplicial complex $\Delta$ is an $i${\it-near-cone} if there exist a sequence of nonempty simplicial
complexes $\Delta = \Delta(0) \supset \Delta(1) \supset \ldots \supset \Delta(i)$ such that for every $1\leq j\leq i$
there is a vertex $v_j \in \Delta(j-1)$ such that $\Delta(j) = {\rm ast}_{\Delta(j-1)}v_j$ and
$\Delta(j-1)$ is a near-cone with respect to $v_j$. The sequence $v_1, \ldots, v_i$ is called the {\it apex} of $\Delta$.
\end{dfn}

$i$-near-cones were first defined by Neve \cite{n}. The following is a simple consequence of its definition.

\begin{lem} \label{face}
Let $\Delta$ be an $i$-near-cone with apex $v_1, \ldots, v_i$, such that ${\rm dim} \Delta\geq 2i-2$. Then $F=\{v_1, v_2, \ldots, v_i\}$ is a face of $\Delta$.
\end{lem}
\begin{proof}
If $i=1$, then there is nothing to prove. So assume that $i\geq 2$. Since ${\rm dim} \Delta\geq 2i-2$, $\Delta$ contains an ($2i-1$)-face, say $\sigma_0$. Considering $\sigma \setminus \{v_1, v_2, \ldots, v_{i-1}\}$, implies that $\Delta$ contains an $i$-face $\sigma_1$, which does not contain the vertices $v_1, v_2, \ldots, v_{i-1}$. Therefore by notations of Definition \ref{ncone}, $\sigma_1\in \Delta(i-1)$. Now $\Delta(i-1)$ is a near-cone with respect to $v_i$, which implies that $\Delta(i-1)$ contains an $i$-face $\sigma_2$ such that $v_i \in \sigma_2$. Since $\sigma_2\in \Delta(i-1)\subset \Delta(i-2)$ and $\Delta(i-2)$ is a near-cone with respect to $v_{i-1}$, the simplicial complex $\Delta(i-2)$ contains an $i$-face $\sigma_3$ such that $\{v_i, v_{i-1}\} \subseteq \sigma_3$. If $i=2$, then we are done. Otherwise repeating the argument above shows that $F=\{v_1, v_2, \ldots, v_i\}$ is a face of $\Delta$.
\end{proof}

\begin{exmp}
Lemma \ref{face} is not true if ${\rm dim} \Delta < 2i-2$. For example assume that
$$\mathcal{F}=\{\{v_1, v_2, v_4, v_6\}, \{v_1,v_3\}, \{v_1,v_5\}, \{v_2,v_3\}, \{v_2,v_5\}, \{v_3, v_4\}, \{v_3,v_5\}, \{v_3,v_6\}\}.$$
Then $\Delta(\mathcal{F})$ is a $3$-near-cone, where $\Delta(0) = \Delta(\mathcal{F})$, $\Delta(1)$ is the simplicial complex generated by $\{\{v_2, v_4, v_6\}, \{v_2,v_3\}, \{v_2,v_5\}, \{v_3, v_4\}, \{v_3,v_5\}, \{v_3,v_6\}\}$ and $\Delta(2)$ is the simplicial complex generated by $\{\{v_4, v_6\}, \{v_3, v_4\}, \{v_3,v_5\}, \{v_3,v_6\}\}$ and $\Delta(3)$ is the simplicial complex generated by $\{\{v_4, v_6\}, \{v_5\}\}$. Now ${\rm dim} \Delta=3 < 4 = 2i-2$ and $F=\{v_1, v_2, v_3\}$ is not a face of $\Delta$.
\end{exmp}

\subsection{Sequentially Cohen-Macaulay complexes and depth}
Let $\mathbb{F}$ be a field. A simplicial complex
$\Delta$ is {\it Cohen-Macaulay} over $\mathbb{F}$ if  $\widetilde{H}_i({\rm lk}_{\Delta} \sigma; \mathbb{F}) =
0$ for all $i < {\rm dim}({\rm lk}_{\Delta} \sigma)$ and all faces $\sigma$ of $\Delta$ (including $\sigma = \emptyset$), where $\widetilde{H}_i(\Delta; \mathbb{F})$ denotes the simplicial homology of $\Delta$ with coefficients in $\mathbb{F}$. It is well-known that every Cohen-Macaulay simplicial complex is pure and that every skeleton of a Cohen-Macaulay simplicial complex is Cohen-Macaulay. A simplicial complex is {\it sequentially Cohen-Macaulay} over $\mathbb{F}$ if every
pure $r$-skeleton of $\Delta$ is Cohen-Macaulay over $\mathbb{F}$. Thus a simplicial complex is Cohen-Macaulay
if and only if it is pure and sequentially Cohen-Macaulay.

Woodroofe \cite{w} defined the {\it depth} of $\Delta$ over $\mathbb{F}$ as
$${\rm depth}_{\mathbb{F}} \Delta = \max\{d : \Delta^{(d)} \ \ {\rm is \ \ Cohen-Macaulay\ \ over} \ \ \mathbb{F}\}.$$
Thus by \cite[Corollary 4.5]{d}, ${\rm depth}_{\mathbb{F}}\Delta$ is the minimum facet dimension of ${\rm Shift}_{\mathbb{F}} \Delta$. We
note that ${\rm depth}_{\mathbb{F}}\Delta$ is one less than the ring-theoretic depth of the
"Stanley-Reisner ring" $\mathbb{F}[\Delta]$ \cite[Theorem 3.7]{sm}. If $\Delta$ is sequentially Cohen-Macaulay over $\mathbb{F}$ then
${\rm depth}_{\mathbb{F}}\Delta$ is the minimum facet dimension of $\Delta$.

By the definition of simplicial homology we have
$\widetilde{H}_d(\Delta^{(d+1)} ; \mathbb{F})=\widetilde{H}_d(\Delta ; \mathbb{F})$. As a simple consequence one obtains the following equivalent characterization:
$${\rm depth}_{\mathbb{F}}\Delta = \max\{d : \widetilde{H}_i({\rm lk}_{\Delta} \sigma; \mathbb{F}) = 0 \ \ {\rm for\ \ all} \ \ \sigma \in \Delta \ \ {\rm and} \ \ i < d-|\sigma|\}.$$
In particular, we notice that ${\rm depth}_{\mathbb{F}}\Delta$ is at most the minimal facet
dimension, since if $\sigma$ is a facet then $\widetilde{H}_{-1}({\rm lk}_{\Delta} \sigma; \mathbb{F})=\widetilde{H}_{-1}(\emptyset; \mathbb{F})= \mathbb{F}$.


\section{A new proof of Borg's result} \label{sec3}

In this section, using  shifting theory, we prove that Conjecture \ref{con} holds for shifted complexes. This will be a new proof for a result of Borg \cite[Theorem 2.7]{b}. Our proof is based on the proof of \cite[Lemma 3.1]{w} due to Woodroofe.

The following Lemma is the main step of our proof. In its proof we do not rely on a specific shifting operator, but only require $(S_1, S_2, S_3, S_4, S_5)$ for the operator Shift.
\begin{lem} \label{shifted}
If $\Delta$ is a shifted complex having
minimal facet cardinality $k$, then $\Delta$ is ($t,r$)-EKR for every natural numbers $t\leq r$ with $k\geq(t+1)(r-t+1)$.
\end{lem}
\begin{proof}
Let $\Delta$ have ordered vertex set $\{v_1, \ldots, v_n\}$, and let $\mathcal{A}$ be a $t$-intersecting $r$-family
 of faces of $\Delta$. Using induction we prove that $|\mathcal{A}|\leq f_{r-t}({\rm lk}_{\Delta}\{v_1, \ldots, v_t\})$. Our base
cases are when $\Delta$ is a simplex (Theorem \ref{ekr}), and the trivial case where
$r = t$.

If $\Delta$ is not a simplex and $r > t$, then by $S_1, S_2, S_3$ and $S_5$, we have
that ${\rm Shift}\mathcal{A}$ is a shifted $t$-intersecting $r$-family  of faces of $\Delta= {\rm Shift}\Delta$
with $|{\rm Shift}\mathcal{A}| = |\mathcal{A}|$. For simplification let $W:=\{v_{n-t+1}, \ldots, v_n\}$. For every subset $U$ of $W$ let $\mathcal{C}_U$ be the set of all faces $\sigma\in {\rm Shift}\mathcal{A}$ with $W\cap \sigma = U$, so that $|\mathcal{A}| = |{\rm Shift}\mathcal{A}| = \sum|\mathcal{C}_U|$, where the sum is taken over all subsets $U$ of $W$. We study $\mathcal{C}_U$ in terms of $|U|$.

{\bf Case 1.} $r < |U|+t$. We claim that in this case $\mathcal{C}_U=\emptyset$. Assume that $\mathcal{C}_U \neq\emptyset$ and choose a member $C\in \mathcal{C}_U$. Since
$$n-|U| > k -|U| \geq(t+1)(r-t+1)-|U|>r+t-|U|,$$  by the definition of shiftedness, in $C$ one can replace the members of $U$  by some other vertices, such that the new vertices do not belong to $C\cup W$. We call this new set $C^{\prime}$. By the choice of $W$ we have $C^{\prime}\in {\rm Shift}\mathcal{A}$ and since $r < |U|+t$, it follows that $|C\cap C^{\prime}| <t$, which is a impossible, because ${\rm Shift}\mathcal{A}$ is a $t$-intersecting family. Thus $\mathcal{C}_U=\emptyset$ for every subset $U$ of $F$ with $r < |U|+t$.

{\bf Case 2.} $r \geq |U|+t$. Let $\mathcal{C}^{\prime}_U =\{ \sigma \setminus U : \sigma \in \mathcal{C}_U\}$. Hence $|\mathcal{C}^{\prime}_U| = |\mathcal{C}_U|$.
 Assume that $|\mathcal{C}^{\prime}_U| \geq 1$ and suppose that $\mathcal{C}^{\prime}_U$ is not $t$-intersecting for some $U\subseteq W$. Then there are $\sigma, \tau \in\mathcal{C}_U$ such that $|(\sigma \setminus U) \cap (\tau \setminus U)|\leq t-1$. Now $|\sigma \cup \tau|\leq 2r-t$ and since
$U \subseteq \sigma \cap \tau$, we conclude that $|\sigma \cup \tau \cup W| \leq 2r-|U|$. By assumption $r\geq t+1$ and $k\geq(t+1)(r-t+1)$ which implies that $k\geq 2r$ and thus
 $|\sigma \cup \tau \cup W| \leq k-|U| < n-|U|$.  It follows that there
exist $|U|$ vertices $v_{\ell_1}, \ldots, v_{\ell_{|U|}}$ such that for every $i$ with $1\leq i \leq |U|$, $v_{\ell_i} \notin \sigma \cup \tau \cup W$.
But then $\tau'=(\tau\setminus U)\cup \{v_{\ell_1}, \ldots, v_{\ell_{|U|}}\}$ is in ${\rm Shift}\mathcal{A}$, by the choice of $W$ and the
the definition of shiftedness, and $|\sigma \cap \tau'|\leq t-1$, which contradicts that
${\rm Shift}\mathcal{A}$ is $t$-intersecting. We conclude that $\mathcal{C}^{\prime}_U$ is a $t$-intersecting ($r-|U|$)-
family of faces of ${\rm lk}_{\Delta_U}U$, where $\Delta_U$ is the simplicial complex obtained from $\Delta$ by deleting the vertices of the set $W\setminus U$. Since ${\rm lk}_{\Delta_U}U$ is a shifted complex on ground set $\{v_1, \ldots, v_{n-t}\}$ with
minimum facet cardinality at least $k-|U|$ and since $k-|U|\geq(t+1)(r-|U|-t+1)$, we conclude that $$|\mathcal{C}_U| = |\mathcal{C}^{\prime}_U|\leq  f_{r-|U|-t}({\rm lk}_{\Delta_U}(U \cup \{v_1, \ldots, v_t\})),$$ by our induction hypothesis.

Finally we have
$$|\mathcal{A}| = |{\rm Shift}\mathcal{A}| = \sum_{U\subseteq W, r \geq |U|+t}|\mathcal{C}_U| \leq \sum_{U\subseteq W, r \geq |U|+t}f_{r-|U|-t}({\rm lk}_{\Delta_U}(U \cup \{v_1, \ldots, v_t\}))$$

$$= f_{r-t}({\rm lk}_{\Delta}\{v_1, \ldots, v_t\}).$$

Where the second equality follows from case 1 and the first inequality follows from case 2.
\end{proof}

Borg \cite{b} proved that Conjecture \ref{con} holds for shifted complexes. Here using Lemma \ref{shifted} we give a new proof for it.

\begin{thm} \label{proof}
{\rm (}\cite[Theorem 2.7]{b}{\rm )} Conjecture \ref{con} is true if $\Delta$ is shifted complex.
\end{thm}
\begin{proof}
Note that $\mathcal{A}$ is the disjoint union of the sets $\mathcal{A}_{(s)}:=\mathcal{A}\cap \Delta_{(s)}$ with $s\in S$. Now by Lemma \ref{shifted} for every $s\in S$ we have
$$|\mathcal{A}_{(s)}|\leq \max_{\tau \in \Delta_{(t)}}f_{s-t}({\rm lk}_{\Delta}\tau)= f_{s-t}({\rm lk}_{\Delta}\sigma),$$
where $\sigma=\{v_1, \ldots, v_t\}$. Thus $|\mathcal{A}| \leq \sum_{s\in S}f_{s-t}({\rm lk}_{\Delta}\sigma)$ and this completes the proof.
\end{proof}

\begin{prop} \label{easy}
Assume that $\Delta$ is a simplicial complex and ${\rm Shift}\Delta$ is its exterior algebraic shift. Let us consider ${\rm Shift}\Delta$ as having ordered vertex set
$\{v_1, \ldots, v_n\}$. Then for every two integers $r\geq t$ we have
$$f_{r-t}({\rm lk}_{{\rm Shift} \Delta}\{v_1, \ldots, v_t\})\geq \max\{f_{r-t}({\rm lk}_{\Delta}\sigma) : \sigma \ \ {\rm is \ \ a \ \ } t{\rm-face \ \ of }\ \ \Delta\}.$$
\end{prop}
\begin{proof}
Let $\tau$ be a $t$-face of $\Delta$ such that for every $t$-face $\sigma$ of $\Delta$, $s:=f_{r-t}({\rm lk}_{\Delta}\tau)\geq f_{r-t}({\rm lk}_{\Delta}\sigma)$. Assume that $\sigma_1, \ldots, \sigma_s$ are $r$-faces of $\Delta$ which contain $\tau$. If $\Gamma$ is the simplicial complex generated by $\sigma_1, \ldots, \sigma_s$, then by $S_4$, ${\rm Shift}\Gamma\subseteq {\rm Shift}\Delta$. Note that $\Gamma$ is a cone with apex $\sigma$ and therefore by \cite[Corollary 5.4]{n}, ${\rm Shift}\Gamma$ is also a cone with an apex set $\sigma^{\prime}$ of cardinality $t$. Hence
$$f_{r-t}({\rm lk}_{\Delta}\tau)=s=f_r(\Gamma)=f_r({\rm Shift}\Gamma)=f_{r-t}({\rm lk}_{{\rm Shift} \Gamma}\sigma^{\prime})\leq f_{r-t}({\rm lk}_{{\rm Shift} \Delta}\sigma^{\prime})$$
$$\leq f_{r-t}({\rm lk}_{{\rm Shift} \Delta}\{v_1, \ldots, v_t\}).$$
\end{proof}

Note that Proposition \ref{easy} essentially says that there exists a shifted simplicial complex for which the right hand side of inequality $(\ast)$, in Conjecture \ref{con}, takes its maximum value. This suggests that shifted simplicial complexes are an easy case of Conjecture \ref{con}.


\section{Intersecting faces of $i$-near-cones} \label{sec4}

In this section we settle Conjecture \ref{con} in the case of $i$-near-cones for some special class of parameters. In the proof we use exterior  algebraic shifting and Lemma \ref{shifted}. Therefore in this section we fix a field $\mathbb{F}$ and by ${\rm Shift}\Delta$ we always mean the exterior algebraic shifting with respect to $\mathbb{F}$. First we need the following proposition, which shows that the the low dimensional skeleta of an $i$-near-cone satisfy a shifting property.

\begin{prop} \label{skl}
Let $\Delta$ be an $i$-near-cone on vertex set $\{v_1, \ldots, v_n\}$, having
minimal facet cardinality $k$. Assume that the sequence $v_1, \ldots v_i$ is the apex of $\Delta$. Then for every $s\leq k-i-1$ and every $\sigma \in \Delta^{(s)}$, $(\sigma\setminus \{v_t\})\cup\{v_j\}\in \Delta^{(s)}$, for every $1\leq j \leq i$ and every $t>j$, with $v_t\in \sigma$.
\end{prop}
\begin{proof}
We use the notations from Definition \ref{ncone}. Assume that $\sigma$ is a face of $\Delta^{(s)}$ such that $v_j\notin \sigma$ and suppose that $\{v_1, \ldots, v_{j-1}\}\cap \sigma=\{v_{i_1}, \ldots, v_{i_m}\}$, with $i_1< i_2 < \ldots < i_m$. Then $\sigma_0:=\sigma \setminus \{v_1, \ldots, v_{j-1}\}$ belongs to $\Delta(j-1)$ and therefore by the definition of $i$-near-cone $\sigma_1=(\sigma_0\setminus\{v_t\})\cup\{v_j\}\in \Delta(j-1)$. Since $\sigma_1$ is contained in a facet of $\Delta$ and
$$|\sigma_1|=|\sigma_0|=|\sigma|-m\leq s+1-m\leq k-i-m,$$
there exist vertices $u_1, \ldots, u_m$ such that for every $1\leq l \leq m$, $u_l\notin \sigma_1\cup\{v_1, \ldots, v_i\}$ and $\sigma_1\cup \{u_1, \ldots, u_m\}$ is a face of $\Delta$. Now by the definition of $i$-near cone $$\sigma_2= (\sigma_1\setminus \{u_m\})\cup\{v_{i_m}\},$$
 $$\sigma_3=(\sigma_2\setminus \{u_{m-1}\})\cup\{v_{i_{m-1}}\},$$
  $$\vdots$$
 and
   $$\sigma_{m+1}= (\sigma_m\setminus \{u_1\})\cup\{v_{i_1}\}=(\sigma\setminus \{v_t\})\cup\{v_j\}$$ are faces of $\Delta$.
\end{proof}

Notice that Proposition \ref{skl}  says
that if $\Delta$ is an $i$-near cone, then $\Delta^{(s)}$ is "shifted with respect to $v_1, \ldots, v_i$", for every $s\leq k-i-1$. Nevo examined the algebraic shift
of a near-cone, showing:

\begin{lem} \label{nevo}
{\rm (}\cite[Corollary 5.3]{n}{\rm )} Assume that $\Delta$ is a near-cone
with apex $v$, let us consider ${\rm Shift}\Delta$ as having ordered vertex set
$\{u_1, \ldots, u_n\}$ and ${\rm Shift}({\rm lk}_{\Delta} v)$ as having ordered vertex set
$\{u_2, \ldots, u_n\}$. Then
$${\rm lk}_{{\rm Shift}\Delta} u_1 = {\rm Shift}({\rm lk}_{\Delta} v).$$
\end{lem}

The following Lemma shows that exterior algebraic shifting commutes with link in the case of low dimensional skeleta of $i$-near-cones.

\begin{lem} \label{link}
Let $\Delta$ be an $i$-near-cone with ${\rm dim} \Delta\geq 2i-2$ having
minimal facet cardinality $k$.
Assume that the sequence $v_1, v_2, \ldots, v_i$ is the apex of $\Delta$. Let $F:=\{v_1, v_2, \ldots, v_i\}$. Consider ${\rm Shift} \Delta$ as having ordered vertex set
$\{u_1, \ldots, u_n\}$ and consider ${\rm Shift}({\rm lk}_{\Delta} F)$ as having ordered vertex set
$\{u_{i+1}, \ldots, u_n\}$. Then for every $s\leq k-i-1$ we have
$${\rm lk}_{{\rm Shift} \Delta^{(s)}}\{u_1, \ldots, u_i\}={\rm Shift}({\rm lk}_{\Delta^{(s)}} F).$$
\end{lem}
\begin{proof}
We prove the lemma by induction on $i$. The case $i=1$ is trivial by Lemma \ref{nevo} and \cite[Corollary 2.4]{w}. So assume that $i\geq 2$. Now
$${\rm lk}_{{\rm Shift} \Delta^{(s)}}\{u_1, \ldots, u_i\}= {\rm lk}_{{\rm lk}_{{\rm Shift} \Delta^{(s)}}\{u_1, \ldots, u_{i-1}\}}u_i.$$
By induction hypothesis
$${\rm lk}_{{\rm Shift} \Delta^{(s)}}\{u_1, \ldots, u_{i-1}\}={\rm Shift}({\rm lk}_{\Delta^{(s)}} \{v_1, \ldots, v_{i-1}\}).$$
Therefore
$${\rm lk}_{{\rm Shift} \Delta^{(s)}}\{u_1, \ldots, u_i\}= {\rm lk}_{{\rm Shift}({\rm lk}_{\Delta^{(s)}} \{v_1, \ldots, v_{i-1}\})}u_i.$$
Now Proposition \ref{skl} implies that ${\rm lk}_{\Delta^{(s)}} \{v_1, \ldots, v_{i-1}\}$ is a near-cone with apex $v_i$. Thus by Lemma \ref{nevo}
$${\rm lk}_{{\rm Shift}({\rm lk}_{\Delta^{(s)}} \{v_1, \ldots, v_{i-1}\})}u_i={\rm Shift}({\rm lk}_{{\rm lk}_{\Delta^{(s)}} \{v_1, \ldots, v_{i-1}\}}v_i)$$
$$={\rm Shift}({\rm lk}_{\Delta^{(s)}} F).$$
\end{proof}

The following proposition is an immediate consequence of Lemma \ref{link} and \cite[Corollary 2.4]{w}.

\begin{prop} \label{fvect}
Let $\Delta$ be an $i$-near cone with ${\rm dim} \Delta\geq 2i-2$ having
minimal facet cardinality $k$.
Assume that the sequence $v_1, v_2, \ldots, v_i$ is the apex of $\Delta$. Let $F:=\{v_1, v_2, \ldots, v_i\}$. Consider ${\rm Shift} \Delta$ as having ordered vertex set
$\{u_1, \ldots, u_n\}$ and consider ${\rm Shift}({\rm lk}_{\Delta} F)$ as having ordered vertex set
$\{u_{i+1}, \ldots, u_n\}$. Then for every $r$ with $r\leq k-2i$
$$f_r({\rm lk}_{{\rm Shift} \Delta}\{u_1, \ldots, u_i\})=f_r({\rm Shift}({\rm lk}_{\Delta} F)).$$
\end{prop}

We are now ready to prove the main result of this section. By applying exterior algebraic shifting we prove:

\begin{thm} \label{main}
If $\Delta$ is an $i$-near-cone, then $\Delta$
is ($i,r$)-EKR for every $i\leq r$ with ${\rm depth}_{\mathbb{F}}\Delta\geq (i+1)(r-i+1)-1$.
\end{thm}
\begin{proof}
The case $r=i$ is trivial. So assume that $r>i$. Then $\dim \Delta \geq {\rm depth}_{\mathbb{F}}\Delta \geq (i+1)(r-i+1)-1\geq 2i+1>2i-2$.
 Therefore using the notations from Definition \ref{ncone}, Lemma \ref{face} implies that $F= \{v_1, \ldots, v_i\}$ is a face of $\Delta$. Let $\mathcal{A}$ be an $i$-intersecting $r$-family of faces of $\Delta$.
We show that $|\mathcal{A}|\leq f_{r-i}({\rm lk}_{\Delta} F)$, where $F$ is the apex of $\Delta$.
Apply algebraic shifting and consider ${\rm Shift}\Delta$ as having ordered vertex set
$\{u_1, \ldots, u_n\}$. ${\rm Shift}\mathcal{A}$ is an $i$-intersecting $r$-family of faces
of ${\rm Shift}_{\mathbb{F}}\Delta$ with $|{\rm Shift}\mathcal{A}| = |\mathcal{A}|$ by $S_3$ and $S_5$. By \cite[Corollary 4.5]{d} and the definition of depth, the minimum facet cardinality of ${\rm Shift}\Delta$
is ${\rm depth}_{\mathbb{F}}\Delta+1$, hence
$$|\mathcal{A}|\leq f_{r-i}({\rm lk}_{{\rm Shift}\Delta} \{u_1, \ldots, u_i\}) = f_{r-i}({\rm lk}_\Delta F),$$
by Lemma \ref{shifted} and Proposition \ref{fvect}. Note that since $r>i$, the assumption implies that $k\geq {\rm depth}_{\mathbb{F}}\Delta+1\geq (i+1)(r-i+1)= i(r-i)+r+1>r+i$, where $k$ is the minimum facet cardinality of $\Delta$. Hence $r-i < k-2i$ and thus Proposition \ref{fvect} is applicable here.
\end{proof}

The following corollary settles Conjecture \ref{con} in the case of $i$-near-cones for some special class of parameters and it is a consequence of Theorem \ref{main} and its proof.

\begin{cor} \label{conje}
Let $i\leq r$ be two integers. Assume that $\Delta$ is an $i$-near-cone such that ${\rm depth}_{\mathbb{F}}\Delta\geq (i+1)(r-i+1)-1$ and suppose that  $S\neq \emptyset$ is a subset of $[i, r]$. Then every $i$-intersecting family $\mathcal{A}$
 of faces of $\Delta$ with $\mathcal{A}\subseteq \bigcup_{s\in S}\Delta_{(s)}$ satisfies the following inequality:
$$|\mathcal{A}| \leq \max \sum_{s\in S}f_{s-i}({\rm lk}_{\Delta}\sigma),$$
where the maximum is taken all over $i$ faces of $\Delta$.
\end{cor}
\begin{proof}
The proof is similar to the proof if Theorem \ref{proof}. Just one should use Theorem \ref{main} instead of Lemma \ref{shifted}.
\end{proof}

The following corollary is an immediate consequence of Corollary \ref{conje} and proves Conjecture \ref{con} in the case of sequentially Cohen-Macaulay
$i$-near-cones for $t=i$. Note that if $\Delta$ is sequentially Cohen-Macaulay over $\mathbb{F}$ then
${\rm depth}_{\mathbb{F}}\Delta$ is the minimum facet dimension of $\Delta$.

\begin{cor} \label{scm}
Let $i\leq r$ be two integers. Assume that $\Delta$ is a sequentially Cohen-Macaulay  $i$-near-cone having
minimal facet cardinality $k\geq (i+1)(r-i+1)$ and suppose that  $S\neq \emptyset$ is a subset of $[i, r]$. Then every $i$-intersecting family $\mathcal{A}$
 of faces of $\Delta$ with $\mathcal{A}\subseteq \bigcup_{s\in S}\Delta_{(s)}$ satisfies the following inequality:
$$|\mathcal{A}| \leq \max \sum_{s\in S}f_{s-i}({\rm lk}_{\Delta}\sigma),$$
where the maximum is taken all over $i$ faces of $\Delta$.
\end{cor}


\section*{Acknowledgments}

This work was done while the author visited Philipps-Universit${\rm \ddot{a}}$t Marburg supported by DAAD. The author thanks
Professor Volkmar Welker for useful discussions during the preparation of the article. He also thanks Professors Peter Borg and Russ Woodroofe for reading an
earlier version of this article and for their helpful comments.



\end{document}